\documentclass{amsart}

\newtheorem{theorem}{Theorem}[section]
\newtheorem{lemma}[theorem]{Lemma}
\newtheorem{fact}[theorem]{Fact}

\theoremstyle{definition}

\theoremstyle{assumption}
\newtheorem{assumption}[theorem]{Assumption}

\theoremstyle{claim}

\theoremstyle{remark}

\newtheorem*{note}{Note}

\numberwithin{equation}{section}

\usepackage{graphicx}

\begin{document}

\title{$\partial$-reducible handle additions}

\author{Han Lou}
\address{School of Mathematical Sciences, Dalian University of Technology, Dalian, People's Republic of China, 116024}
\email{lh0720@mail.dlut.edu.cn}
%
\author{Mingxing Zhang}
\address{School of Mathematical Sciences, Dalian University of Technology, Dalian, People's Republic of China, 116024}
\email{zhangmx@dlut.edu.cn}




\begin{abstract}
Let $M$ be a simple 3-manifold, and $F$ be a component of $\partial M$ of genus at least 2. Let $\alpha$ and $\beta$ be separating slopes on $F$. Let $M(\alpha)$ (resp. $M(\beta)$) be the manifold obtained by adding a 2-handle along $\alpha$ (resp. $\beta$). If $M(\alpha)$ and $M(\beta)$ are $\partial$-reducible, then the minimal geometric intersection number of $\alpha$ and $\beta$ is at most 8.
\end{abstract}

\maketitle

%


\section{Introduction}
Let $M$ be a compact 3-manifold. $M$ is \textit{irreducible} if each 2-sphere in $M$ bounds a 3-ball. $M$ is \textit{$\partial$-irreducible} if $\partial M$ is incompressible. $M$ is \textit{anannular} if $M$ does not contain essential annuli. $M$ is \textit{atoroidal} if $M$ does not contain essential tori.
A compact, orientable 3-manifold $M$ is said to be \textit{simple} if it is irreducible, $\partial$-irreducible, anannular and atoroidal. By Thurston's Geometrization Theorem, a Haken manifold $M$ is hyperbolic if and only if $M$ is simple.

Let $M$ be a compact 3-manifold, and $F$ be a component of $\partial M$. A \textit{slope} $\alpha$ on $F$ is an isotopy class of essential simple closed curves on $F$. Denote by $\Delta(\alpha, \beta)$ the minimal geometric intersection number among all the curves representing the slopes $\alpha$ and $\beta$. For a slope $\alpha$ on $F$, denote by $M(\alpha)$ the manifold obtained by attaching a 2-handle to $M$ along a regular neighborhood of $\alpha$ on $F$, then capping off a possible 2-sphere component of the resulting manifold by a 3-ball. Note that if $F$ is a torus, then $M(\alpha)$ is the Dehn filling along $\alpha$. Given a simple 3-manifold $M$, a slope $\alpha$ is \textit{degenerating} if $M(\alpha)$ is non-simple.

It is known that there are finitely many degenerating Dehn fillings. This result is precisely stated on the upper bound of $\Delta(\alpha, \beta)$ for two degenerating slopes $\alpha$ and $\beta$. There are 10 cases as follows. C. Gordon and J. Luecke showed that $\Delta(\alpha, \beta)\le 1$ if $M(\alpha)$ and $M(\beta)$ are reducible \cite{GL1}. M. Scharlemann showed that $\Delta(\alpha, \beta)=0$ if $M(\alpha)$ is reducible and $M(\beta)$ is $\partial$-reducible \cite{Sch1}. R. Qiu and Y. Wu independently showed that $\Delta(\alpha, \beta)\le 2$ if $M(\alpha)$ is reducible and $M(\beta)$ is annular \cite{Qiu1}, \cite{Wu2}.
Y. Wu and S. Oh independently showed that $\Delta(\alpha, \beta)\le 3$ if $M(\alpha)$ is reducible and $M(\beta)$ is toroidal \cite{Wu1}, \cite{Oh}.
Y. Wu showed that $\Delta(\alpha, \beta)\le 1$ if $M(\alpha)$ and $M(\beta)$ are  $\partial$-reducible \cite{Ying2}. C. Gordon and Y. Wu showed that $\Delta(\alpha, \beta)\le 2$ if $M(\alpha)$ is $\partial$-reducible and $M(\beta)$ is annular \cite{Gordon}. C. Gordon and J. Luecke showed that $\Delta(\alpha, \beta)\le 2$ if $M(\alpha)$ is $\partial$-reducible and $M(\beta)$ is toroidal \cite{GL2}. C. Gordon showed that $\Delta(\alpha, \beta)\le 5$ if $M(\alpha)$ and $M(\beta)$ are annular, $\Delta(\alpha, \beta)\le 5$ if $M(\alpha)$ is annular and $M(\beta)$ is toroidal and  $\Delta(\alpha, \beta)\le 8$ if $M(\alpha)$ and $M(\beta)$ are toroidal \cite{Gordon3}.

Suppose that the genus of $F$ is at least two. M. Scharlemann and Y. Wu gave an example that has infinitely many degenerating slopes \cite{Wu}. Let $\alpha$ be a degenerating slope on $F$. If either $\alpha$ is separating or there are no separating degenerating curves coplanar with $\alpha$, $\alpha$ is a \textit{basic} degenerating slope. M. Scharlemann and Y. Wu proved that there are finitely many basic degenerating slopes. Furthermore, $\Delta(\alpha, \beta)$ for two separating degenerating slopes $\alpha$ and $\beta$ has an upper bound. Specially, $\Delta(\alpha,\beta)=0$ if $M(\alpha)$ is reducible and $M(\beta)$ is $\partial$-reducible \cite{Wu}. R. Qiu and M. Zhang showed that $\Delta(\alpha, \beta)\le 2$ if $M(\alpha)$ and $M(\beta)$ are both reducible \cite{Qiu}.

The case that $M(\alpha)$ and $M(\beta)$ are both $\partial$-reducible was also considered by Y. Li, R. Qiu and M. Zhang. They showed that $\Delta(\alpha, \beta)\le 2$ when the $\partial$-reducing disks have boundaries on $\partial M-F$ \cite{Xing2}. In this paper, we consider the case with no limits on the $\partial$-reducing disks.

\begin{theorem}
\label{main}
Suppose that $M$ is a simple 3-manifold and $F$ is a component of $\partial M$ of genus at least 2. Let $\alpha$ and $\beta$ be separating slopes on $F$ such that $M(\alpha)$ and $M(\beta)$ are $\partial$-reducible. Then $\Delta(\alpha, \beta)\le 8$.
\end{theorem}

\section{Preliminaries}

Suppose that $M$ is a simple 3-manifold, and $F$ is a component of $\partial M$ of genus at least two. Let $\alpha$ and $\beta$ be two separating slopes on $F$ such that $M(\alpha)$ and $M(\beta)$ are both $\partial$-reducible. Note that if one of $M(\alpha)$ and $M(\beta)$ is reducible, then $\Delta(\alpha, \beta)=0$. Thus we can assume that $M(\alpha)$ and $M(\beta)$ are irreducible. We denote by $H_\alpha$ (resp. $H_\beta$) the 2-handle attached to $M$ along $\alpha$ (resp. $\beta$) to obtain $M(\alpha)$ (resp. $M(\beta)$). Let $\hat P$ (resp. $\hat Q$) be an essential disk in $M(\alpha)$ (resp. $M(\beta)$) such that $\lvert\hat P\cap H_{\alpha}\rvert$ (resp. $\lvert\hat Q\cap H_{\beta}\rvert$) is minimal among all essential disks in $M(\alpha)$ (resp. $M(\beta)$). Let $P=\hat P\cap M$ and $Q=\hat Q\cap M$. Obviously,
$\hat P\cap H_{\alpha}$ (resp. $\hat Q\cap H_{\beta}$) consists of disks with boundary components having slope $\alpha$ (resp. $\beta$).

\begin{lemma}
\label{incompressible}
$P$ (resp. $Q$) is incompressible and $\partial$-incompressible in $M$.
\end{lemma}
\begin{proof}
Suppose that $P$ is compressible in $M$. Let $D$ be a compressing disk for $P$ in $M$. Then $\partial D$ bounds a disk $D^\prime$ in $\hat P$. $D\cup D^\prime$ is a 2-sphere which bounds a 3-ball in $M(\alpha)$ since $M(\alpha)$ is irreducible. By isotoping $D^\prime$ to $D$, we reduce $\lvert\hat P\cap H_{\alpha}\rvert$, contradicting that $\lvert\hat P\cap H_{\alpha}\rvert$ is minimal. Thus $P$ is incompressible in $M$.

Suppose that $P$ is $\partial$-compressible. Let $D$ be a $\partial$-compressing disk for $P$ in $M$. $\partial D=u\cup v$, $u\subset F$ and $v\subset P$.

Case 1: $v$ connects two components of $\partial P-\partial\hat P$.

After $\partial$-compressing $P$ along $D$, we can get an essential disk $\hat P_1$ in $M(\alpha)$.

Case 2: $\partial v$ is in a component of $\partial P-\partial\hat P$.

By $\partial$-compressing $P$ along $D$, we get two planar surface $P_1$ and $P_2$. Suppose that $\partial\hat P\subset\partial P_2$. Then we get an essential disk $\hat P_1$ by capping off the components of $\partial P_1$ having slope $\alpha$.

Case 3: $\partial v$ is in $\partial\hat P$.

Let $P_1$ and $P_2$ be planar surfaces obtained by $\partial$-compressing $P$ along $D$. After capping off the components of $\partial P_1$ having slope $\alpha$, we get an essential disk $\hat P_1$.

Case 4: $v$ connects $\partial\hat P$ and one component of $\partial P-\partial\hat P$.

Let $P_1$ be the planar surface obtained by $\partial$-compressing $P$ along $D$. After capping off the components of $\partial P_1$ having slope $\alpha$, we get a disk $\hat P_1$ and $\partial\hat P_1=\partial\hat P$ since $\alpha$ is a trivial slope in $\partial M(\alpha)$. Then $\hat P_1$ is an essential disk in $M(\alpha)$.

In each case we have  $\lvert\hat P_1\cap H_{\alpha}\rvert<\lvert\hat P\cap H_{\alpha}\rvert$, contradicting that $\lvert\hat P\cap H_\alpha\rvert$ is minimal.

Thus $P$ is $\partial$-incompressible in $M$.
\end{proof}

We may assume that $\lvert P\cap Q\rvert$ is minimal. Then each component of $P\cap Q$ is either an essential arc or an essential simple closed curve on both $P$ and $Q$.

R. Litherland and C. Gordon first used the results in graph theory to study Dehn fillings. See \cite{Litherland2}, \cite{Litherland}. Then M. Scharlemann and Y. Wu introduced these methods to study handle additions. See \cite{Wu}, \cite{Sch2}.

Let $\Gamma_P$ be a graph on $\hat P$. The vertices of $\Gamma_P$ are components of $\partial P$ and the edges of $\Gamma_P$ are arc components of $P\cap Q$. Each component of $\partial P-\partial\hat P$ is called an \textit{interior vertex} of $\Gamma_P$. An edge incident to $\partial\hat P$ is called a \textit{boundary edge} of $\Gamma_P$. Similarly, we can define $\Gamma_Q$ in the disk $\hat Q$. Note that by taking $\partial\hat P$ (resp. $\partial\hat Q$) also as a vertex $\Gamma_P$ (resp. $\Gamma_Q$) is a spherical graph.

By Lemma \ref{incompressible}, we have the following Lemma \ref{1-sided}.

\begin{lemma}
\label{1-sided}
There are no 1-sided disk faces in both $\Gamma_P$ and $\Gamma_Q$.
\end{lemma}

\begin{lemma}
\label{paralleledge}
There are no edges that are parallel in both $\Gamma_P$ and $\Gamma_Q$.
\end{lemma}
\begin{proof}
See Lemma 2.1 in \cite{Wu}.
\end{proof}

\begin{lemma}
\label{3q}
There are no $3q$ parallel edges in $\Gamma_P$.
\end{lemma}
\begin{proof}
Suppose that there are $3q$ parallel edges in $\Gamma_P$. Let $\Gamma^\prime$ be the subgraph of $\Gamma_Q$ formed by these edges and all the vertices of $\Gamma_Q$. Let $e$ be the number of edges of $\Gamma^\prime$ and $f$ be the number of disk faces of $\Gamma^\prime$. By Euler characteristic formula, $q-e+f\ge 1$, then $f\ge e-q+1$.

Suppose that there are no 2-sided disk faces of $\Gamma^\prime$. Then $2e\ge 3f$. Thus $\frac{2}{3}e\ge f\ge e-q+1$. Then $e\le 3q-3$, contradicting with $e=3q$. Now there is a 2-sided disk face of $\Gamma^\prime$. The two edges bounding this disk face are parallel in both $\Gamma_P$ and $\Gamma_Q$. This contradicts Lemma \ref{paralleledge}.
\end{proof}

Let $F^\prime$ be the component of $F-(\partial P-\partial\hat P)$ such that $\partial\hat P\subset F^\prime$. Number the components of $\partial P-\partial\hat P$ as $\partial_1P$, $\partial_2P$,...,$\partial_uP$,...,$\partial_pP$ consecutively on $F$ such that $\partial_1P=\partial F^\prime$. This means that $\partial_uP$ and $\partial_{u+1}P$ bound an annulus in $F$ with interior disjoint from $P$. See Figure \ref{7}. Similarly, number the components of $\partial Q-\partial \hat Q$ as $\partial_1Q$, $\partial_2Q$,...,$\partial_iQ$,...,$\partial_qQ$. These give the corresponding labels of the vertices of $\Gamma_P$ and $\Gamma_Q$. Since $M$ is simple, then $p,q>1$.

\begin{figure}
\centering
\includegraphics[scale=0.35]{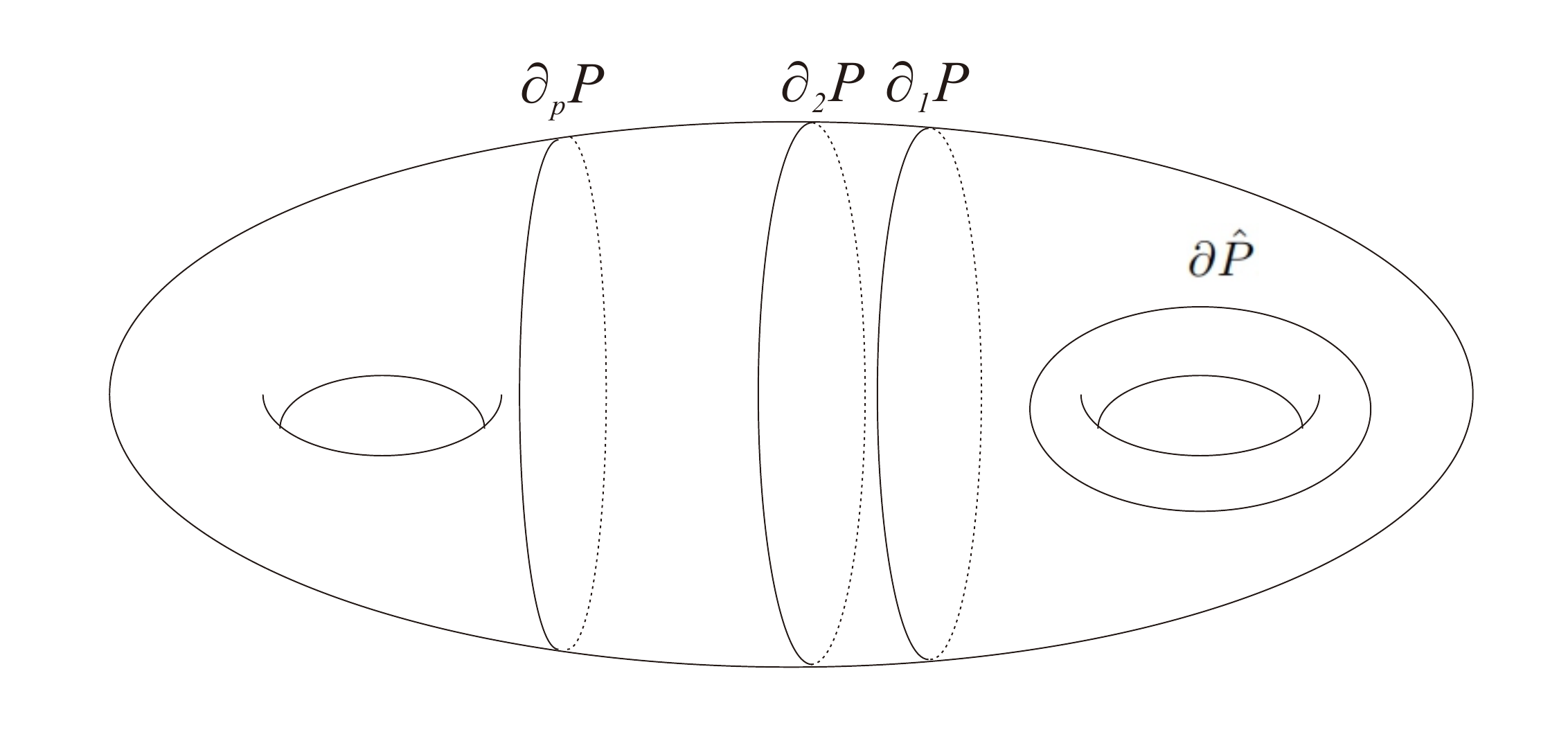}
\caption{}
\label{7}
\end{figure}

Let $x$ be an endpoint of an arc in $P\cap Q$. If $x$ belongs to $\partial_uP\cap\partial_iQ$, then we label it as $(u,i)$. If $x$ belongs to $\partial\hat{P}\cap\partial_iQ$, then we label it as $(*,i)$. If $x$ belongs to $\partial_uP\cap\partial \hat Q$, then we label it as $(u,*)$. If $x$ belongs to $\partial\hat{P}\cap\partial \hat Q$, then we label it as $(*,*)$. On $\Gamma_P$, the labels are written only by the second index $i$ (or $*$) for short. And on $\Gamma_Q$, the labels are written only by the first index $u$ (or $*$) for short. Here, $*$ is called the \textit{boundary label}. Around each vertex of $\Gamma_P$, the labels appear as $q,q-1,\cdots,1,*,\cdots,*,1,2,\cdots,q$ in clockwise direction or anticlockwise direction. By giving a sign to the numbered labels as in \cite{Xing}, we can assume that

\begin{assumption}
\label{assump}
  $-q, -(q-1), \cdots, -1, *, \cdots, *, +1, +2, \cdots, +q$ (resp. $-p, -(p-1), \cdots, -1, *, \cdots, *, +1, +2, \cdots, +p$) appear in clockwise direction around each vertex of $\Gamma_P$ (resp. $\Gamma_Q$).
\end{assumption}

\begin{note} By taking $\Gamma_P$ as a graph on disk $\hat P$, the labels $-q$,$-(q-1)$,$\cdots$,$-1$,$*$,$\cdots$,$*$,$+1$,
$+2$,$\cdots$,$+q$ appear on $\partial \hat P$ in anticlockwise direction. See Figure \ref{8}.
\end{note}

\begin{figure}
\centering
\includegraphics[scale=0.35]{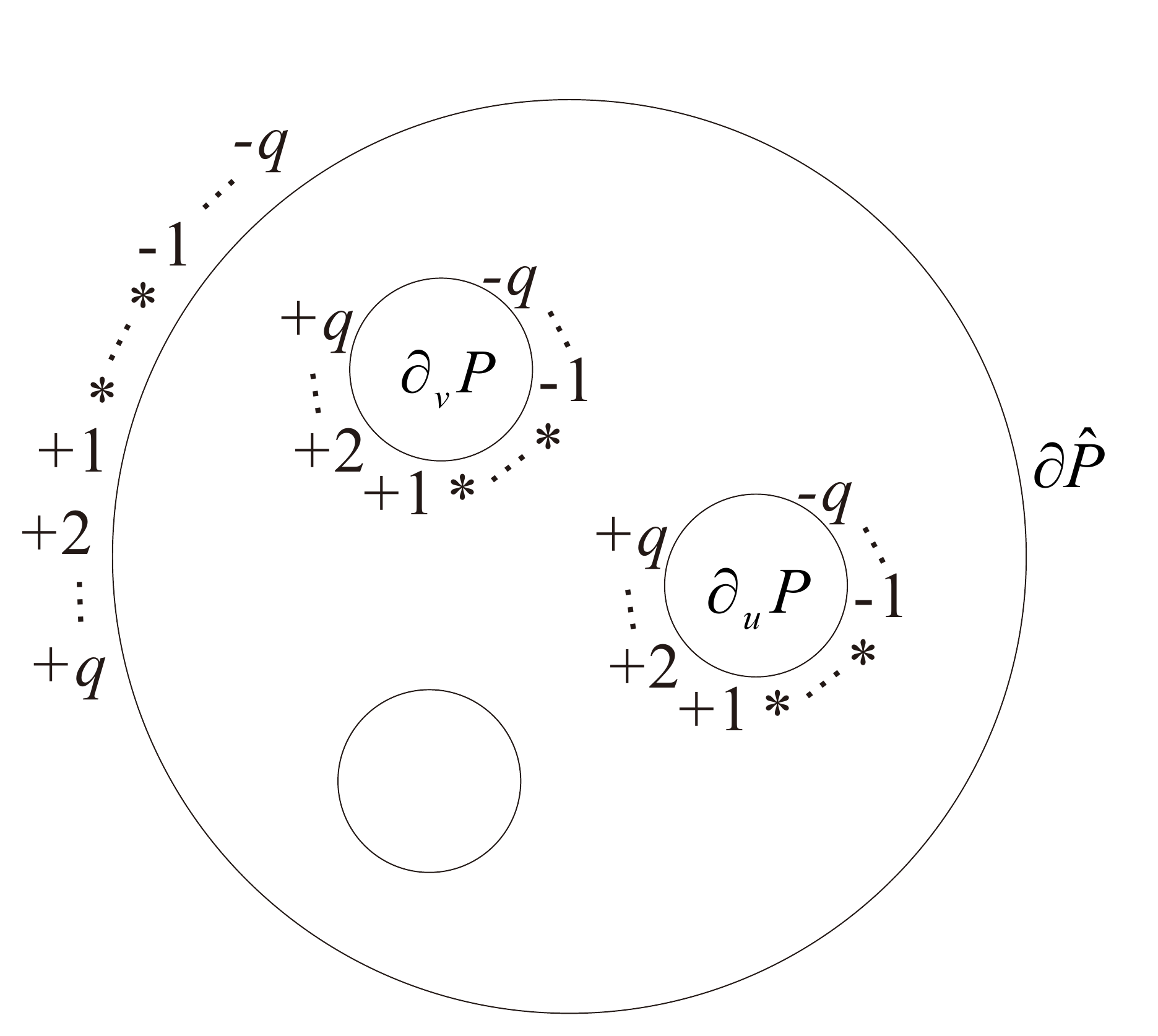}
\caption{}
\label{8}
\end{figure}

On $\Gamma_P$ (resp. $\Gamma_Q$), each edge has a \textit{label pair} $(i,j)$ of its two endpoints, $i,j\in\{+1,+2,\cdots,+q,-q,\cdots,-1,*\}$. As Lemma 3.3 in \cite{Xing}, we have a parity rule.

\begin{lemma}
\label{label}
Let $e$ be an edge in $\Gamma_P$ (resp. $\Gamma_Q$). $e$ cannot have label pair $(i,i)$, $i\in\{+1,+2,\cdots, +q,-q,\cdots,-1\}$.
\end{lemma}

Let $x$ be a signed label in $\{+1, +2,..., +q, -q,..., -1\}$. An \textit{$x$-edge} is an edge in $\Gamma_P$ with label $x$ at one endpoint. Let $B_P^x$ denote the subgraph of $\Gamma_P$ consisting of all the vertices of $\Gamma_P$ and all the $x$-edges. Let $x$ be a signed label in $\{+1, +2,..., +p, -p,..., -1\}$, the definition for $B_Q^x$ is the same.

A cycle of $B_P^x$ which bounds a disk face of $\Gamma_P$ and contains no edges with boundary labels (at the endpoints) is called a \textit{virtual Scharlemann cycle}. In a virtual Scharlemann cycle, the label pair of each edge is the same, which is called the \textit{label pair} of the virtual Scharlemann cycle. A virtual Scharlemann cycle with label pair $(i,j)$ is called a \textit{Scharlemann cycle} if $i\ne -j$.

As in \cite{Qiu}, we have the following Lemma \ref{opposite} and Lemma \ref{Scharlemann}.

\begin{lemma}
\label{opposite}
Suppose $S=\{e_i\mid i=1,2,...,n\}$ is a set of parallel edges in $\Gamma_P$. Each edge in $S$ has no boundary labels. If one of the edges, say $e_k$, has opposite labels at its two endpoints, then each edge in $S$ has opposite labels at its two endpoints.
\end{lemma}

\begin{lemma}
\label{Scharlemann}
Let $x\in\{+1,+2,...,+q,-q,...,-1\}$. Let $D$ be a disk face of $B_P^x$ and no edges in $D$ have boundary labels. Then there is a virtual Scharlemann cycle lying in $D$.
\end{lemma}

\begin{lemma}
\label{compress}
There are no Scharlemann cycles in $\Gamma_P$ or $\Gamma_Q$.
\end{lemma}
\begin{proof}
See Lemma 2.5.2 (a) in \cite{Culler}.
\end{proof}

The following Lemma \ref{valency} is Lemma 2.6.5 in \cite{Culler}.
\begin{lemma}
\label{valency}
Let $\Gamma$ be a graph in a disk $D$ with no trivial loops or parallel edges, such that every vertex of $\Gamma$ belongs to a boundary edge. Then $\Gamma$ has a vertex of valency at most 3 which belongs to a single boundary edge.
\end{lemma}

\section{Proof of theorem 1.1}
Let $R_P=\Delta(\alpha,\partial \hat Q)$, $\Delta =\Delta(\alpha, \beta)$.
\begin{lemma}
\label{diskface}
Suppose that $\Delta\ge 6$. For any label $v\in\{+1,..., +p, -p,..., -1\}$, there are at least $(\frac{\Delta}{2}q+\frac{R_P}{2}-3q+3)$ 2-sided disk faces of $B_Q^v$.
\end{lemma}
\begin{proof}
Denote the number of edges in $B_Q^v$ by $e$. Then $e=\frac{\Delta}{2}q+\frac{R_P}{2}$. Denote the number of 2-sided disk faces of $B_Q^v$ by $f_1$ and the number of disk faces with at least 3 sides by $f_2$. By Euler characteristic  formula, $q-e+f_1+f_2\ge 1$. Since $2f_1+3f_2\le 2e$, $f_2\le\frac{2e-2f_1}{3}$. Then $ q-e+f_1+\frac{2e-2f_1}{3}\ge 1$. Thus $f_1\ge e-3q+3=\frac{\Delta}{2}q+\frac{R_P}{2}-3q+3$.
\end{proof}

Let $D$ be a 2-sided disk face of $B_Q^v$. See Figure \ref{2}. The edges of $\Gamma_Q$ in $D$ are parallel. Suppose that these edges are incident to two subarcs $z$ and $z^\prime$ of $\partial Q$. When we go from top to bottom along $z$ (resp. from bottom to top along $z^\prime$), the labels appear in the direction of $+1$,$+2$,$\cdots$,$+p$,$-p$,$\cdots$,$-1$,$*$. Let $X(D)$ (resp. $Y(D)$) be the collection of labels in $z$ (resp. $z^\prime$). There are three cases of 2-sided disk face $D$ of $B_Q^v$. In case 1, two labels $v$ are both in $X(D)$. Note that it is the same after rotating the figure by $\pi$ if two labels $v$ are both in $Y(D)$. See Figure \ref{2} (1). In case 2, the top label (resp. lowest label) in $z$ (resp. $z^\prime$) is $v$. See Figure \ref{2} (2). In case 3, the lowest label (resp. top label) in $z$ (resp. $z^\prime$) is $v$. See Figure \ref{2} (3).

We have the following Fact \ref{fact}.

\begin{fact}
\label{fact}
There are no edges in $int D$ with label $v$.
\end{fact}

\begin{figure}
\centering
\includegraphics[scale=0.34]{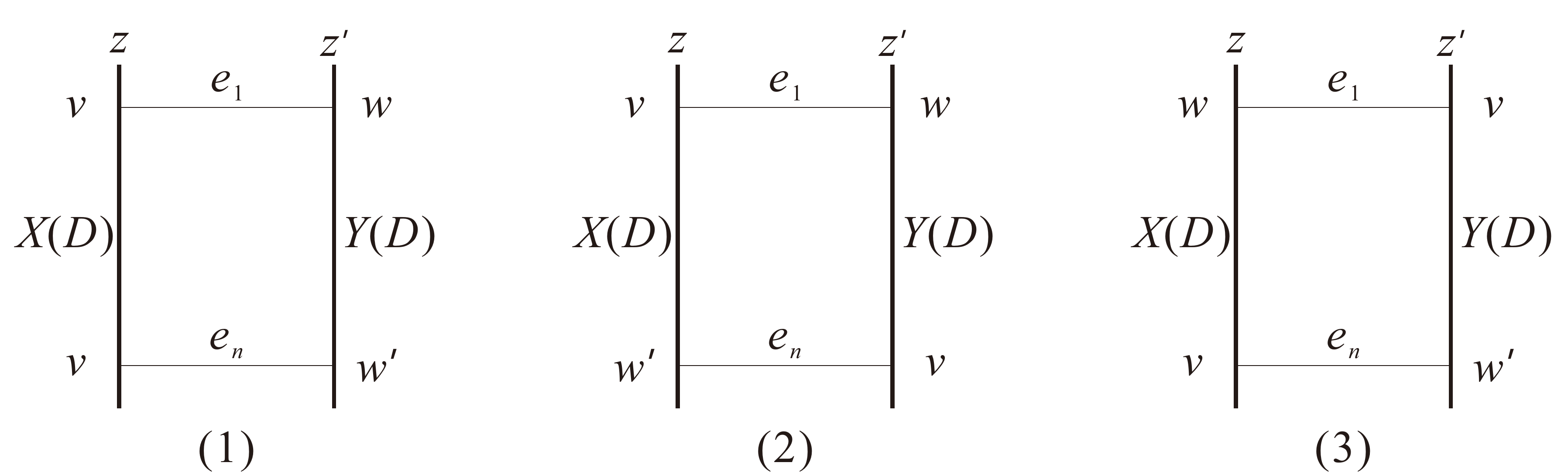}
\caption{}
\label{2}
\end{figure}

If a interior vertex of $\Gamma_P$ is incident to boundary edges or edges with both endpoints in it, then the vertex is \textit{good}. Otherwise the vertex is \textit{bad}.

\begin{lemma}
\label{1}
Suppose that $\Delta\ge 6$. Vertex 1 in $\Gamma_P$ is good.
\end{lemma}
\begin{proof}
Let $D$ be a 2-sided disk face of $B_Q^{+1}$. Then $D$ is as in Figure \ref{firstfig}. Denote the two edges in $\partial D$ by $e_1$ and $e_n$ respectively. If one of $w$ and $w^\prime$ is a boundary label, then this lemma holds. Thus we assume that both $w$ and $w^\prime$ are not boundary labels.

\begin{figure}
\centering
\includegraphics[scale=0.34]{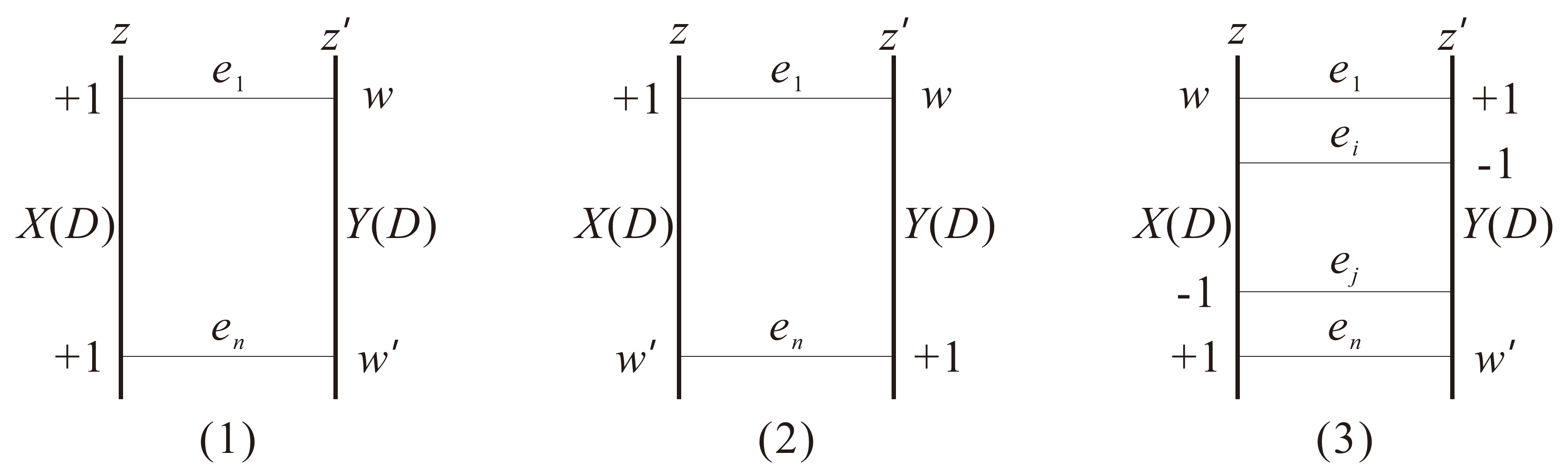}
\caption{}
\label{firstfig}
\end{figure}

Case 1: $D$ is as in Figure \ref{firstfig} (1).

If there are boundary labels in $Y(D)$, then there is label $+1$ in $Y(D)$ by Assumption \ref{assump}. If there are no boundary labels in $Y(D)$, then there is label $+1$ in $Y(D)$ since there are at least $2p+1$ edges in $D$ . By Lemma \ref{label}, $w, w^\prime\ne +1$. Then there is an edge with label $+1$ in $int D$, contradicting with Fact \ref{fact}.

Case 2: $D$ is as in Figure \ref{firstfig} (2).

If there are boundary labels in $X(D)$ or $Y(D)$, then there is an edge with label $+1$ in $int D$ by Assumption \ref{assump} and Lemma \ref{label}, contradicting with Fact \ref{fact}. Thus there are no boundary labels in both $X(D)$ and $Y(D)$. By Lemma \ref{Scharlemann} and Lemma \ref{compress}, there is a virtual Scharlemann cycle $\Sigma$ with label pair $(+1,-1)$ or $(+p,-p)$ in $D$. By Lemma \ref{opposite}, $e_1$ has label pair $(+1,-1)$. Then $e_1$ is an edge with both endpoints in vertex 1 in $\Gamma_P$.

Case 3: $D$ is as in Figure \ref{firstfig} (3).

There must be label $-1$ in both $X(D)$ and $Y(D)$ since $w$ and $w^\prime$ are not boundary labels. Assume that edge $e_i$ (resp. $e_j$) has label $-1$ in $Y(D)$ (resp. $X(D)$). If $e_j$ is higher than $e_i$, $e_i$ and $e_j$ must have boundary labels by Assumption \ref{assump}. Then $e_i$ and $e_j$ are boundary edges incident to vertex 1 in $\Gamma_P$, which means that vertex 1 is good. Now assume that $e_i$ is higher than $e_j$. See Figure \ref{firstfig} (3). $e_i$ and $e_j$ bound a 2-sided disk face $D^\prime$ of $B_Q^{-1}$. By Fact \ref{fact}, the edges in $D^\prime$ have no boundary labels. Then by Lemma \ref{Scharlemann} and Lemma \ref{compress}, there is a virtual Scharlemann cycle $\Sigma$ with label pair $(+1,-1)$ or $(+p,-p)$ in $D^\prime$. By Lemma \ref{opposite}, $e_i$ has label pair $(+1,-1)$, contradicting with Fact \ref{fact}.

Thus vertex 1 in $\Gamma_P$ is good.
\end{proof}

\begin{lemma}
\label{4edges} For a bad vertex $v$ ($v>1$) of $\Gamma_P$, let $D$ be a 2-sided disk face of $B_Q^{+v}$. Then there are labels $+1$ and $-1$ in both $X(D)$ and $Y(D)$.
\end{lemma}
\begin{proof}
Since $v$ is bad, then the 2-sided disk face of $B_Q^{+v}$ is as in Figure \ref{2}, where $w$ and $w^\prime$ are not boundary labels.

Case 1: $D$ is as in Figure \ref{2} (1).

By Assumption \ref{assump} there are labels $+1$ and $-1$ in $X(D)$.

Suppose that there are no boundary labels in $Y(D)$. Since there are at least $2p+1$ edges in $D$, then there is label $+v$ in $Y(D)$, contradicting with Fact \ref{fact} or Lemma \ref{label}. Thus there are boundary labels in $Y(D)$. Then there are labels $+1$ and $-1$ in $Y(D)$ by Assumption \ref{assump}.

Case 2: $D$ is as in Figure \ref{2} (2).

Suppose that there are no boundary labels in both $X(D)$ and $Y(D)$. By Lemma \ref{Scharlemann} and Lemma \ref{compress}, there is a virtual Scharlemann cycle with label pair $(+1,-1)$ or $(+p,-p)$ in $D$. By Lemma \ref{opposite}, $e_1$ has label pair $(+v,-v)$. Then $e_1$ is an edge with both endpoints in vertex $v$ in $\Gamma_P$, contradicting that vertex $v$ is bad.

Now there are boundary labels in $X(D)$ or $Y(D)$. Without loss of generality, we assume that there are boundary labels in $X(D)$. Thus there are labels $+1$ and $-1$ in $X(D)$ by Assumption \ref{assump}. Then there are at least $(2p-v+3)$ edges in $D$. See Figure \ref{9} (1). By counting the labels in $Y(D)$ from bottom to top, label $-1$ must appear in $Y(D)$. Since $w$ is not $*$, label $+1$ must appear in $Y(D)$.

Case 3: $D$ is as in Figure \ref{2} (3).

Suppose that there are no boundary labels in both $X(D)$ and $Y(D)$. Similar to the proof in Case 2, we get a contradiction. Now suppose that there are boundary labels in $X(D)$. Thus there are labels $+1$ and $-1$ in $X(D)$. Then there are at least $(v+2)$ edges in $D$. See Figure \ref{9} (2). By a similar argument as in Case (2), labels $+1$ and $-1$ must appear in $Y(D)$.
\end{proof}

\begin{figure}
\centering
\includegraphics[scale=0.35]{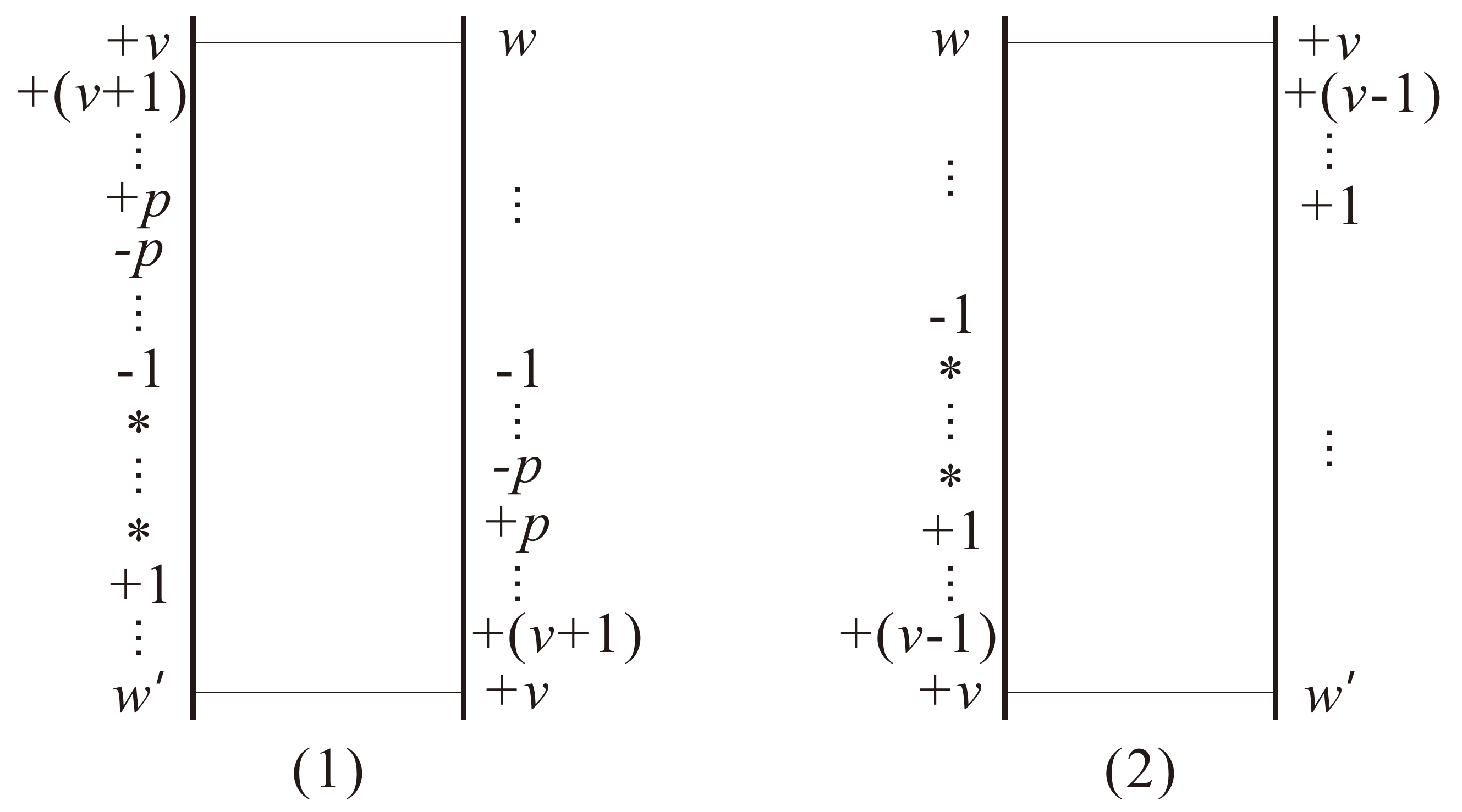}
\caption{}
\label{9}
\end{figure}

\begin{lemma}
\label{parallel}
Suppose $v$ ($v>1$) is a bad vertex of $\Gamma_P$. Then any two 2-sided disk faces of $B_Q^{+v}$ are not parallel.
\end{lemma}
\begin{proof}
We first prove that any two 2-sided disk faces $D_1$ and $D_2$ of $B_Q^{+v}$ cannot share a common edge.
Suppose not. See Figure \ref{3}. Without loss of generality, assume that $X(D_1)\cap X(D_2)=\{+v\}$. By Lemma \ref{4edges}, there is an edge $e_1$ in $D_1$ with label $-1$ in $Y(D_1)$ and an edge $e_2$ in $D_2$ with label $+1$ in $Y(D_2)$. By Assumption \ref{assump}, there is an edge $e_3$ between $e_1$ and $e_2$ with label $+v$ in $Y(D_1)\cup Y(D_2)$. By Lemma \ref{label}, $w\ne +v$. Thus $e_3$ is in $int D_1$ or $int D_2$, contradicting with Fact \ref{fact}.

\begin{figure}
\centering
\includegraphics[scale=0.35]{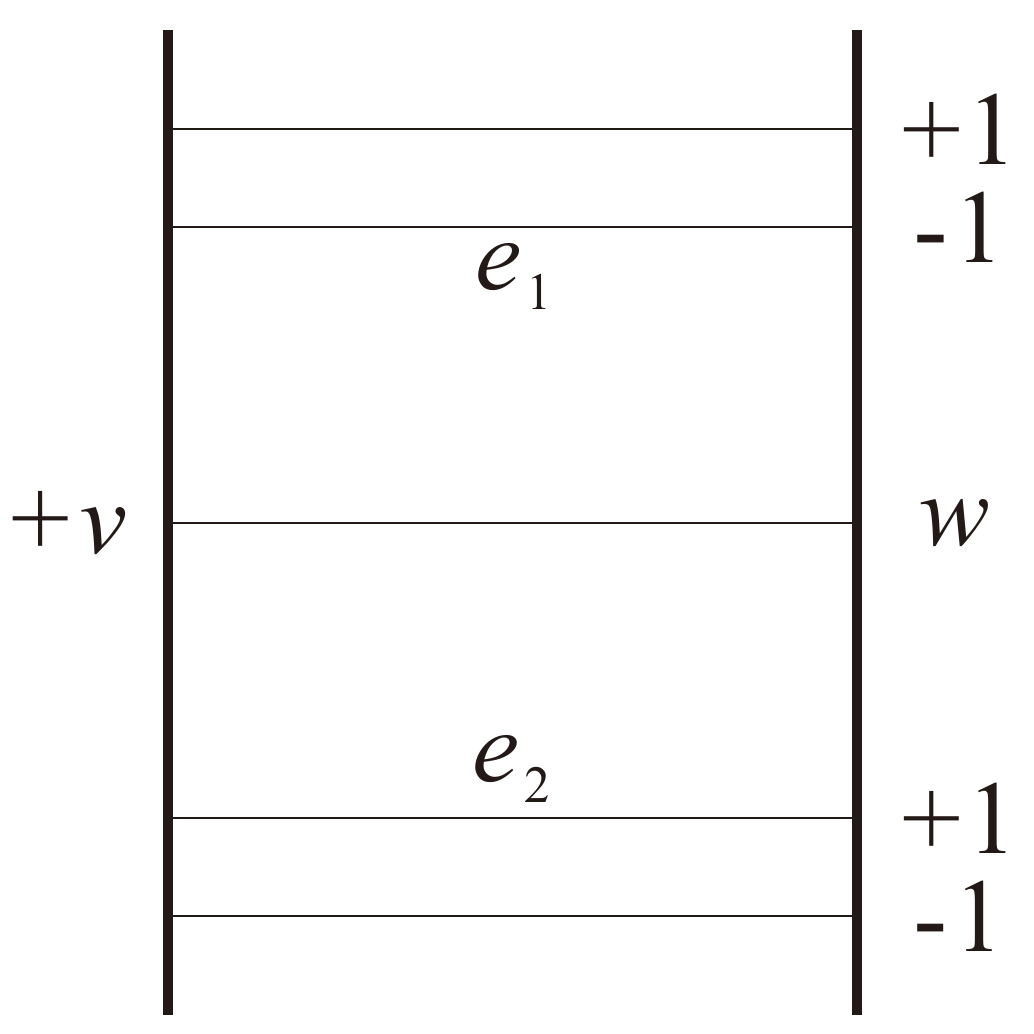}
\caption{}
\label{3}
\end{figure}

Now suppose that there are two parallel but not adjacent 2-sided disk faces $D_1$ and $D_2$ of $B_Q^{+v}$. See Figure \ref{5}. There must be another 2-sided disk face $D_3$ which shares a common edge with $D_2$, a contradiction. Hence this lemma holds.

\begin{figure}
\centering
\includegraphics[scale=0.35]{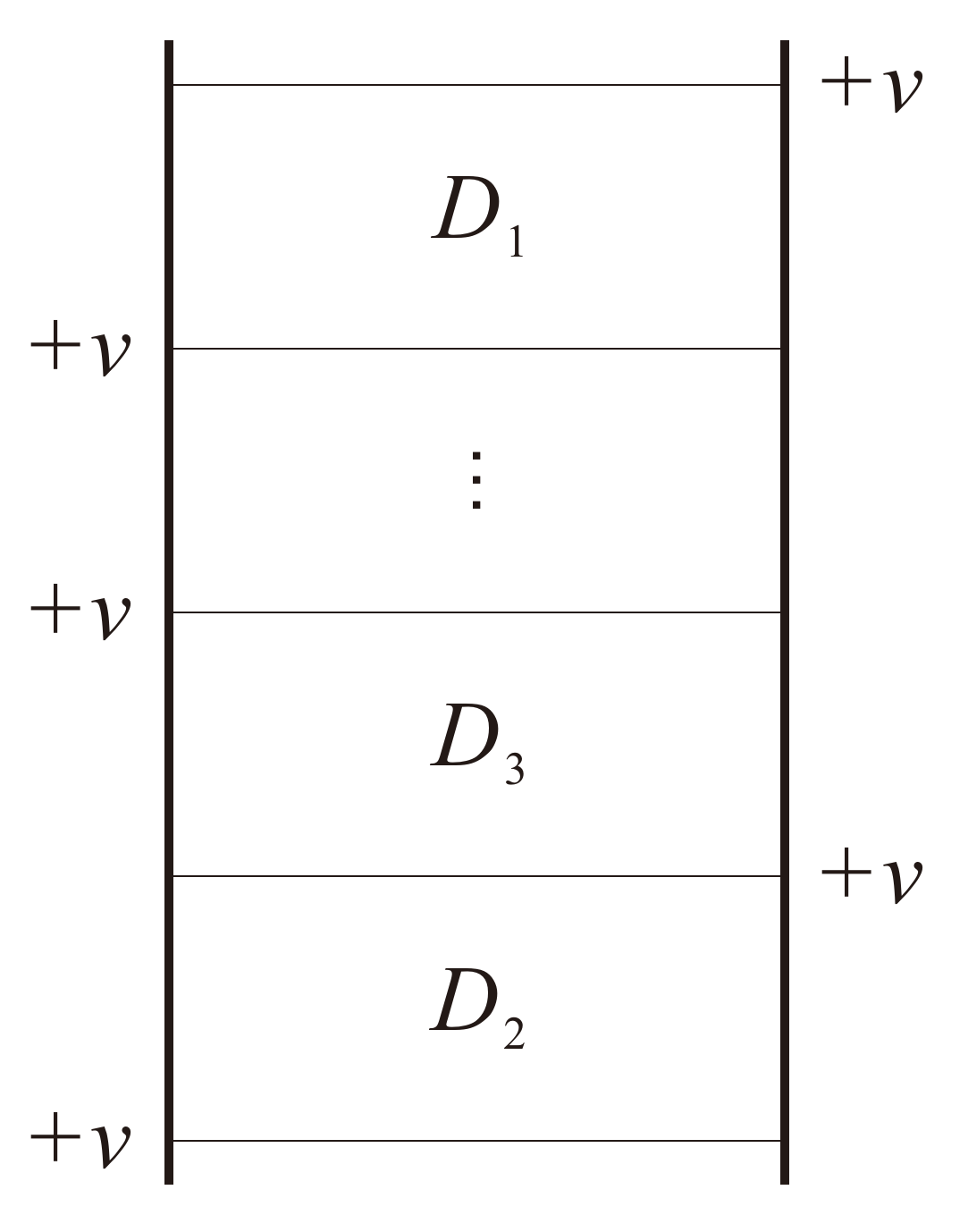}
\caption{}
\label{5}
\end{figure}

\end{proof}

\begin{lemma}
\label{good}
Suppose that $\Delta\ge 10$. Then each interior vertex in $\Gamma_P$ is good.
\end{lemma}
\begin{proof}
Suppose otherwise that there is a bad vertex $v$ of $\Gamma_P$. By Lemma \ref{1}, $v>1$. By taking each 2-sided disk face of $B_Q^{+v}$ as an edge, we have a graph $D_Q^{+v}$. Denote the number of edges in $D_Q^{+v}$ by $e$ and the number of disk faces by $f$. By Lemma \ref{diskface}, $e\ge\frac{\Delta}{2}q+\frac{R_P}{2}-3q+3\ge 5q+\frac{R_P}{2}-3q+3=2q+\frac{R_P}{2}+3$. By Lemma \ref{parallel}, there are no 2-sided disk faces of $D_Q^{+v}$.

Suppose that there are no 3-sided disk faces of $D_Q^{+v}$. Then $4f\le 2e$, that is $f\le\frac{e}{2}$. By Euler characteristic formula, $q-e+f\ge 1$, $e-q+1\le f\le\frac{e}{2}$. Hence $e\le 2q-2$, a contradiction.

Now suppose that $D$ is a 3-sided disk face of $D_Q^{+v}$ bounded by 2-sided disk faces $D_1$, $D_2$ and $D_3$. Suppose that $z$, $z^\prime$ and $z^{\prime\prime}$ are three subarcs of $\partial Q$ which contain the endpoints of all the edges in $D$, $D_1$, $D_2$ and $D_3$. See Figure \ref{6}. Suppose that there is an edge with label $+v$ in $int D$. Then there are two 2-sided disk faces of $B_Q^{+v}$ sharing a common edge, contradicting with Lemma \ref{parallel}. Thus there are no edges in $int D$ with label $+v$. Then $D$ is a 3-sided disk face of $B_Q^{+v}$ bounded by edges $e_1$, $e_2$ and $e_3$.

\begin{figure}
\centering
\includegraphics[scale=0.4]{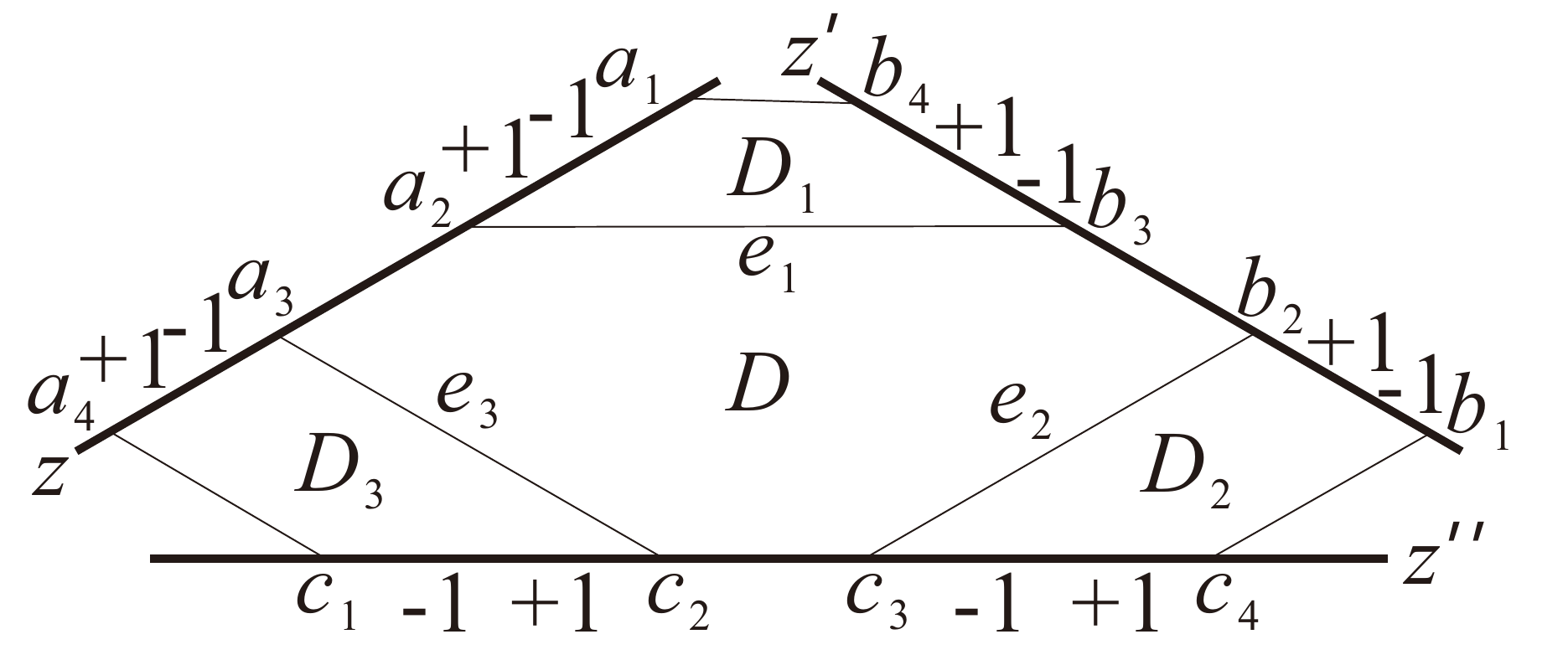}
\caption{}
\label{6}
\end{figure}

By Lemma \ref{4edges}, there are labels $+1$ and $-1$ in both $X(D_i)$ and $Y(D_i)$, $i=1,2,3$. Then at least one of $a_2$ (resp. $b_2$, $c_2$) and $a_3$ (resp. $b_3$, $c_3$) is $+v$ since there are no edges with label $+v$ in $int D$ and $int D_i$, $i=1,2,3$. We claim that exactly one of $a_2$ and $a_3$ is $+v$. Otherwise, both $a_2$ and $a_3$ are $+v$. By Lemma \ref{label}, $b_3, c_2\ne +v$. Thus both $b_2$ and $c_3$ are $+v$, contradicting with Lemma \ref{label}. Similarly, exactly one of $b_2$ (resp. $c_2$) and $b_3$ (resp. $c_3$) is $+v$.

Suppose that there is an edge in $int D$ with label $+1$ between $a_2$ and $a_3$. Then there are three labels $+1$ in $z$. By Assumption \ref{assump}, there are two labels $+v$ between $a_1$ and $a_4$. Then both $a_2$ and $a_3$ are labels $+v$, which contradicts with the last paragraph. Thus no edges in $int D$ have label $+1$ between $a_2$ and $a_3$. Similarly, no edges in $int D$ have label $-1$ between $a_2$ and $a_3$. By the same argument, no edges in $int D$ have label $+1$ or $-1$ between $b_2$ and $b_3$ and between $c_2$ and $c_3$.

Suppose that there is a boundary label $*$ between $a_2$ and $a_3$. Since vertex $v$ is bad, $a_i$, $b_i$ and $c_i$ are not boundary label $*$, $i=1,2,3,4$. Then $a_2$ is $-1$ and $a_3$ is $+1$, contradicting that exactly one of $a_2$ and $a_3$ is $+v$. Thus there are no edges in $int D$ with boundary label $*$ between $a_2$ and $a_3$. The same is true for $b_2$, $b_3$ and $c_2$, $c_3$.

Now $D$ is a disk face of $B_Q^{+v}$ which contains no boundary labels. By Lemma \ref{Scharlemann} and lemma \ref{compress}, there is a virtual Scharlemann cycle $\Sigma$ with label pair $(+1,-1)$ or $(+p,-p)$ in $D$. By Lemma \ref{opposite}, $e_1$ has label pair $(+v,-v)$. Then $e_1$ is an edge with two endpoints in vertex $v$ in $\Gamma_P$, a contradiction. Thus each vertex of $\Gamma_P$ is good.
\end{proof}

{\bfseries The proof of Theorem \ref{main}.}

 Suppose that $\Delta\ge 10$. By Lemma \ref{good}, each vertex of $\Gamma_P$ is good, i.e, it incident to boundary edges or edges with both endpoints in it. Suppose that there is a vertex $v$ incident to an edge $e$ with both endpoints in it. Then $e$ bounds a disk $D$ in $\hat P$. If there is a vertex $v^\prime$ in $D$, then $v'$ also has an edge $e'$ with both endpoints in it. And $e'$ also bounds a disk $D^\prime$ in $\hat P$. By repeating the above process, we finally get an innermost 1-sided disk face, which contradicts Lemma \ref{1-sided}. Thus each vertex of $\Gamma_P$ is incident to boundary edges.

Let $\overline{\Gamma}_P$ be the reduced graph of $\Gamma_P$. By Lemma \ref{valency}, there is a vertex $v$ of valency at most 3 in $\overline{\Gamma}_P$. Then there are at most three families of parallel edges incident to $v$ in $\Gamma_P$. Since $\Delta\ge 10$, one of them has more than $3q$ edges, contradicting with Lemma \ref{3q}. Thus $\Delta\le 8$.
We complete the proof of Theorem \ref{main}.

\bibliographystyle{amsplain}

\end{document}